\documentclass[a4paper, 12pt]{article}
\usepackage{mathptmx}       
\usepackage{helvet}         
\usepackage{courier}        
\usepackage{type1cm}        
\usepackage{makeidx}         
\usepackage{graphicx}        
\usepackage{multicol}        
\usepackage[bottom]{footmisc}
\usepackage{url}
\usepackage{amsmath}
\usepackage{mathrsfs}
\usepackage{amsfonts}
\usepackage{booktabs}
\usepackage{authblk}

\newtheorem{theorem}{Theorem}[section]
\newtheorem{lemma}[theorem]{Lemma}

\newtheorem{proposition}[theorem]{Proposition}
\newtheorem{definition}[theorem]{Definition}
\newtheorem{assumptions}[theorem]{Assumption}

\DeclareMathOperator{\diff}{d}
\DeclareMathOperator{\lsc}{lsc}
\DeclareMathOperator{\cl}{cl}
\DeclareMathOperator{\interier}{int}
\DeclareMathOperator{\dom}{dom}

\newcommand{\fF}{\frak F}
\newcommand{\fG}{\frak G}
\newcommand{\fS}{\frak S}
\newcommand{\cM}{\mathscr M}

\newcommand{\calF}{\mathcal F}
\newcommand{\calG}{\mathcal G}
\newcommand{\calS}{\mathcal S}

\newcommand{\R}{\mathbb{R}}

\newcommand{\Irn}{\int_{\mathbb R^n}}

\makeindex             

\begin{document}

\title{Strong Duality of Linear Optimisation Problems over Measure Spaces}
\author[1]{Raphael Hauser\thanks{Andrew Wiles Building, Radcliffe Observatory Quarter, Woodstock Road, Oxford OX2 6GG,
United Kingdom, hauser@maths.ox.ac.uk}}
\author[2]{Sergey Shahverdyan\thanks{Andrew Wiles Building, Radcliffe Observatory Quarter, Woodstock Road, Oxford OX2 6GG,
United Kingdom, shahverdyan@maths.ox.ac.uk}}
\affil[1]{Mathematical Institute, University of Oxford}
\affil[2]{Mathematical Institute, University of Oxford}
\date{}
\maketitle

\section{Introduction}\label{problemFormulation}
In \cite{hauser} the authors introduce a general duality result for linear optimisation problems over signed measures with infinitely many constraints in the form of integrals of functions with respect to the decision variables (the measure in question). In this work we present two particular cases of the general duality result for which strong duality holds. In the first case the optimisation problems are over measures with $L^p$ density functions with $1 < p < \infty$. In the second case we consider a semi-infinite optimisation problem where finitely many constraints are given in form of bounds on integrals. The latter case has a particular importance in practice where the model can be applied in robust risk management and model-free option pricing, e.g. \cite{hauser, sergeyRFOP}.\\
In the next section we present the general duality result first introduced in \cite{hauser}. In Section \ref{sec:shapiro} we introduce results on conic linear optimisation problems from \cite{shap}. In Sections \ref{lpDensities} and \ref{semiInf} we use the results from Section \ref{sec:shapiro} to prove duality for two cases described above: measures that have $L^p$ density functions for $1 < p < \infty$ and semi-infinite problems with special structure.

\section{Problem Formulation}\label{problemFormulation}
Let $(\Phi,\fF)$, $(\Gamma,\fG)$ and $(\Sigma,\fS)$ 
be complete measure spaces, and let 
$A:\,\Gamma\times\Phi\rightarrow\R$,
$a:\,\Gamma\rightarrow\R$, 
$B:\,\Sigma\times\Phi\rightarrow\R$, 
$b:\,\Sigma\rightarrow\R$, and 
$c:\,\Phi\rightarrow\R$ 
be bounded measurable functions on these spaces and the corresponding product 
spaces. Let $\cM_{\fF}$, $\cM_{\fG}$ and $\cM_{\fS}$ be the set of signed 
measures with finite variation on $(\Phi,\fF)$, $(\Gamma,\fG)$ and 
$(\Sigma,\fS)$ respectively. 
We now consider the following pair of optimisation problems over $\cM_{\fF}$ and 
$\cM_{\fG}\times\cM_{\fS}$ respectively, which authors show to be duals of each other, 
\begin{align*}
\text{(P')}\quad\sup_{\calF\in\cM_{\fF}}\,&\int_{\Phi}c(x)\diff \calF(x)\\
\text{s.t. }&\int_{\Phi}A(y,x)\diff \calF(x)\leq a(y),\quad(y\in\Gamma),\\
&\int_{\Phi}B(z,x)\diff \calF(x)= b(z),\quad(z\in\Sigma),\\
&\calF\geq 0,
\end{align*}
and 
\begin{align*}
\text{(D')}\quad\inf_{(\calG,\calS)\in\cM_{\fG}\times\cM_{\fS}}
\,&\int_{\Gamma}a(y)\diff \calG(y)+\int_{\Sigma}b(z)\diff\calS(z),\\
\text{s.t. }&\int_{\Gamma}A(y,x)\diff\calG(y)+\int_{\Sigma}
B(z,x)\diff\calS(z)\geq c(x),\quad(x\in\Phi),\\
&\calG\geq 0.
\end{align*}

\begin{theorem}[Weak Duality]\label{weak}
For every (P')-feasible measure $\calF$ and every 
(D')-feasible pair $(\calG,\calS)$ we have 
\begin{equation*}
\int_{\Phi}c(x)\diff\calF(x)\leq \int_{\Gamma}a(y)\diff\calG(y)+
\int_{\Sigma}b(z)\diff\calS(z).
\end{equation*}
\end{theorem}

\begin{proof}
Using Fubini's  Theorem, we have 
\begin{align*}
\int_{\Phi}c(x)\diff\calF(x)&\leq
\int_{\Gamma\times\Phi}A(y,x)\diff(\calG\times\calF)(y,x)+
\int_{\Sigma\times\Phi}B(z,x)\diff(\calS\times\calF)(z,x)\\
&\leq\int_{\Gamma}a(y)\diff\calG(y)+
\int_{\Sigma}b(z)\diff\calS(z).
\end{align*}
\end{proof}
\\
We are interested in finding conditions on measures that imply strong duality between the primal and dual problems, i.e. $Val(P') = Val(D')$, where by $Val(P')$ and $Val(D')$ we denote the optimal values of the problems $(P')$ and $(D')$ respectively.

\section{General Results for Conic Optimisation Problems}\label{sec:shapiro}
As mentioned above, in this section we introduce results on conic linear optimisation problems from \cite{shap}.\\
Consider a conic linear optimisation problem of the following form
\begin{equation}\label{Proof By Shapiro Results:eq1}
\text{(P)  }\min_{f\in C}\langle c, f \rangle \quad \text{subject to} \quad \mathcal Af+h \in K,
\end{equation}
where $X$ and $Y$ are linear spaces, $C\subset X$, $K\subset Y$ are convex cones, $h\in Y$ and $\mathcal A:X\rightarrow Y$ is a linear map. Assume that $X$ and $Y$ are paired with some linear spaces $X'$ and $Y'$ respectively, so bilinear forms $\langle \cdot, \cdot, \rangle:X'\times X\rightarrow \R$ and $\langle \cdot, \cdot, \rangle:Y'\times Y\rightarrow \R$ are defined. We call the problem \eqref{Proof By Shapiro Results:eq1} the Primal problem.\\
The results of Shapiro are based on conjugate duality first introduced by Rockefellar \cite{rock}, \cite{rock-1}.\\
Define the positive dual cone of $C$ as
\begin{equation}\label{Proof By Shapiro Results:eq2}
C^* := \left\{ f^*\in X^* : \langle f^*, f\rangle \geq 0, \quad\forall f \in C \right\},
\end{equation}
and similarly for the cone $K$
\begin{equation}\label{Proof By Shapiro Results:eq3}
K^* := \left\{ g^*\in Y^* : \langle g^*, g\rangle \geq 0, \quad\forall g \in K \right\},
\end{equation}
We also need an assumption for $X'$ so that the adjoint mapping of $A$ exists.\\
\begin{assumptions}\label{Proof By Shapiro Results:asmp1}
For any $g^*\in Y'$ there exists a unique $f^*\in X'$ such that $\langle g^*, \mathcal Af\rangle = \langle f^*, f\rangle$ for all $f\in X$.
\end{assumptions}
Based on this assumption we can define the adjoint mapping $\mathcal A^*:Y'\rightarrow X'$ by the equation
\begin{equation}\label{Proof By Shapiro Results:eq4}
\langle g^*, \mathcal Af\rangle = \langle \mathcal A^*g^*, f\rangle, \quad \forall f\in X.
\end{equation}
Now consider the Lagrangian function of the primal problem \eqref{Proof By Shapiro Results:eq1}
\begin{equation}\label{Proof By Shapiro Results:eq5}
L(f, g^*):= \langle c, f\rangle + \langle g^*, Af + h\rangle
\end{equation}
and the following optimisation problem
\begin{equation}\label{Proof By Shapiro Results:eq6}
\min_{f\in C}\left\{\psi(f):=\max_{g^*\in -K^*}\left\{L(f, g^*)\right\}\right\}
\end{equation}
By changing the $\min$ and $\max$ operators we get the Lagrangian Dual problem
\begin{equation}\label{Proof By Shapiro Results:eq6}
\max_{g^*\in -K^*}\left\{\phi(g^*):=\min_{f\in C}\left\{L(f, g^*)\right\}\right\}
\end{equation}
which is equivalent to the following optimisation problem
\begin{equation}\label{Proof By Shapiro Results:eq7}
\text{(D)  }\max_{g^*\in-K^*}\langle g^*, h\rangle \quad\text{subject to}\quad \mathcal A^*g^*+c\in C^*.
\end{equation}
which we call the dual problem.\\
The aim of this section is to find conditions under which $Val(D) = Val(P)$, where by $Val(D)$ and $Val(P)$ we denote the objective values of the Dual \eqref{Proof By Shapiro Results:eq7} and Primal \eqref{Proof By Shapiro Results:eq1} problems respectively.\\
Note that the dual problem is also a conic linear problem, and that it is easy to show the weak duality, i.e., $Val(D) \leq Val(P)$.\\
Further, we associate with the primal problem the optimal value function
\begin{equation}\label{Proof By Shapiro Results:eq8}
v(g) := \inf\left\{\langle c, f\rangle:f\in C, \mathcal Af + g\in K\right\}. 
\end{equation}
We define $v(g)$ to be $+\infty$ if the set $\left\{f\in C: \mathcal Af + g\in K\right\}$ is empty. We have $Val(P) = v(h)$.\\
From \cite{rock-1} we know that the extended optimal value function $v(g)$ is convex and positively homogeneous of degree $1$, i.e., $\forall t > 0$ and $g\in Y$ $v(tg) = tv(g)$.\\
The conjugate of $v(g)$ is defined as 
\begin{equation}\label{Proof By Shapiro Results:eq9}
v^*(g^*) := \sup_{g\in Y}\left\{\langle g^*, g\rangle - v(g)\right\}.
\end{equation}
Evaluating the formulae above we get
\begin{align*}
v^*(g^*) &= \sup\left\{ \langle g^*, g \rangle - \langle c, f\rangle : (f, g^*)\in X\times Y^*, f\in C, \mathcal Af+g\in K\right\}\\
&= \sup_{f\in C}\sup_{\mathcal Af+g\in K}\left\{ \langle g^*, g \rangle - \langle c, f\rangle \right\}\\
&= \sup_{f\in C}\sup_{g\in K}\left\{ \langle g^*, g -\mathcal Af \rangle - \langle c, f\rangle \right\}\\
&= \sup_{f\in C}\sup_{g\in K}\left\{ \langle g^*, g\rangle - \langle \mathcal A^*g^* + c, f\rangle \right\}
\end{align*}
It is easy to show that $v^*(g^*)$ is the indicator function of the feasible set of the dual problem, since from $g^*\in -K^*$ and $\mathcal A^*g^* + c\in C^*$ follows that $v^*(g^*) = 0$, and $v^*(g^*) = +\infty$ otherwise. So we can write the dual problem as
\begin{equation}\label{Proof By Shapiro Results:eq10}
\max_{g^*\in Y^*}\left\{ \langle g^*, h\rangle - v^*(g^*)\right\}
\end{equation}
Taking the biconjugate of $v(y)$
\begin{equation}\label{Proof By Shapiro Results:eq10}
v^{**}(g):= \sup_{g^*\in Y^*}\left\{ \langle g^*, g\rangle - v^*(g^*)\right\},
\end{equation}
we see that $Val(D) = v^{**}(h)$, hence we get, that if $v(h) = v^{**}(h)$, then there is no duality gap between Lagrangian primal and dual problems.\\
Now we aim to find conditions such that $v(h) = v^{**}(h)$.\\
We described the main approach to the proof of the strong duality in this framework, and now, without going into details, we will introduce the results given by Shapiro. More interested reader can refer to \cite{shap} for more details.\\
We make an assumption which will be considered to be hold throughout this section.
\begin{assumptions}
The spaces $Y$ and $Y'$ are paired locally convex topological vector spaces.
\end{assumptions}
Denote by $\lsc v$ the lower semicontinous hull of the function $v$, i.e.
\begin{equation}\label{Proof By Shapiro Results:eq11}
\lsc v(g) = \min\left\{ v(g), \liminf_{z\rightarrow g}v(z)\right\},
\end{equation}
and by $\cl v$ the closure of the function $v$:
\begin{equation}\label{Proof By Shapiro Results:eq12}
\cl v(\cdot) := \begin{cases}
\lsc v(\cdot), & \mbox{if } \lsc v(g) > -\infty \mbox{ for all } g\in Y,\\ 
-\infty, & \mbox{if } \lsc v(g) = -\infty \mbox{ for at least one } g\in Y.
 \end{cases}
\end{equation}
We say that the problem $(P)$ is sub-consistent if $\lsc v(h) < +\infty$ (if the problem $(P)$ is consistent, i.e. it's feasible set is nonempty, then it is also sub-consistent). Moreover, the Fenchel-Moreau theorem implies that $v^{**} = \cl v$. Taking into account the fact that if $\lsc v(h) < +\infty$ then $\cl v(h) = \lsc v(h)$ we get the following proposition:
\begin{proposition}[Proposition 2.2, \cite{shap}]
The following holds:
\begin{enumerate}
\item $Val(D) = \cl v(h)$.
\item If $(P)$ is sub-consistent, then $Val(D) = \lsc v(h)$.
\end{enumerate}
\end{proposition}
The above proposition shows that if $P$ is sub-consistent then strong duality holds, i.e., $Val(D) = Val(P)$ iff $v(h)$ is lower semicontinous at $g = h$. But it may be difficult to verify the semicontinuity directly, so we seek more tractable conditions in the subsequent analysis.\\
Define the sub-differential of the function $v$ at a point $g$ (where $v$ is finite) as
\begin{equation}\label{Proof By Shapiro Results:eq13}
\partial v(g):= \left\{ g^*\in Y' : v(z) - v(g)\geq \langle g^*, z - g\rangle, \quad \forall z\in Y\right\}.
\end{equation}
We say that $v$ is sub-differentiable at a point $g$ if $v(g)$ is finite and $\partial v(g)$ is nonempty. Further, we know that if $v$ is sub-differentiable at $g = h$, then $v^{**} = v(h)$, and conversely; if $v(b)$ is finite and $v^{**} = v(h)$, then $\partial v(h) = \partial v^{**}(h)$ \cite{rock-1}.\\
We now get the following proposition.
\begin{proposition}[Proposition 2.5, \cite{shap}]
If the optimal value function $v(g)$ is sub-differentiable at the point $g = h$, then $Val(P) = Val(D)$ and the set of optimal solutions of $(D)$ is $\partial v(h)$. Conversely, if $Val(P) = Val(D)$ and is finite, then $Sol(D) = \partial v(h)$. By $Sol(D)$ we denote the set of solutions of the problem $(D)$.
\end{proposition}
However, checking the sub-differentiability for the optimal value function may still be difficult.\\
Consider the set
\begin{equation}\label{Proof By Shapiro Results:eq14}
M := \left\{ (g, \alpha)\in Y\times\R : g = k - \mathcal Af, \alpha\geq \langle c, f \rangle, f\in C, k\in K\right\}.
\end{equation}
It is easy to show that the optimal value of the problem $(P)$ is equal to the optimal value of the following problem
\begin{align*}
&\min \alpha\\ 
\text{ s.t. } &(h, \alpha)\in M.
\end{align*}
\begin{proposition}[Proposition 2.6, \cite{shap}]
Suppose that $Val(P)$ is finite and the cone $M$ is closed in the product topology of $Y\times \R$. Then $Val(P) = Val(D)$ and the primal problem $(P)$ has an optimal solution.
\end{proposition}
From convex analysis we know that if $v(g) < \infty$ and continuous at $h$, $Y$ is a Banach space and $Y^*$ is its standard dual, then $\partial v(h)$ is closed and bounded in the dual topology of $Y^*$ (\cite{hol}, p. 84). We get the following proposition.
\begin{proposition}[Proposition 2.7, \cite{shap}]
If the optimal value function $v(g)$ is continuous at $g = h$ and if $Y$ is a Banach space paired with its standard dual $Y^*$, then the set of optimal solutions of $(D)$ is bounded in the dual norm topology of $Y^*$.
\end{proposition}
From \cite{rob} we know that if $X$ and $Y$ are Banach spaces equipped with strong topologies, the cones $C$ and $K$ are closed and $\langle c, \cdot\rangle$ and $\mathcal A:X\rightarrow Y$ are continuous, then $v(g)$ is continuous at $g = h$ if and only if
\begin{equation}\label{Proof By Shapiro Results:eq15}
h\in \interier(\dom v).
\end{equation}
Since $\dom v = -\mathcal A(C) + K$, we can write the condition \eqref{Proof By Shapiro Results:eq15} as
\begin{equation}\label{Proof By Shapiro Results:eq16}
-h\in \interier(\mathcal A(C) - K).
\end{equation}
Hence we get the final stone of the framework we need to prove our strong duality.
\begin{proposition}[Proposition 2.9, \cite{shap}]\label{mainProp}
Suppose that $X$ and $Y$ are Banach spaces, the cones $C$ and $K$ are closed, $\langle c, \cdot\rangle$ and $\mathcal A:X\rightarrow Y$ are continuous, and that Condition \eqref{Proof By Shapiro Results:eq16} holds. Then $Val(P) = Val(D)$ and $Sol(D)$ is nonempty and bounded.
\end{proposition}
If the cone $K$ has a non-empty interior, then Condition (\ref{Proof By Shapiro Results:eq16}) is equivalent to (\cite{bon-shap}, Proposition 2.106)
\begin{equation}
\exists \bar f\in C \text{ such that } \mathcal A\bar f + h\in \interier(K).
\end{equation}
If the later condition holds, then it is said that the generalized Slater condition is satisfied for Problem (\ref{Proof By Shapiro Results:eq1}).\\
In many applications we have equality type constraints for optimisation problems of the form (\ref{Proof By Shapiro Results:eq1}). In this case the cone $K$ has obviously a single element $0$, and hence, the interior is empty.\\
If the constraint in the problem (\ref{Proof By Shapiro Results:eq1}) is given by the equality
\begin{equation*}
\mathcal Af + h = 0,
\end{equation*}
then the regularity condition (\ref{Proof By Shapiro Results:eq16}) is equivalent to (\cite{bon-shap}, section 2.3.4)
\begin{eqnarray*}
&\mathcal A(X) = Y,\\
&\exists \bar f\in \interier(C) \text{ s. t. } \mathcal A\bar f + h = 0.
\end{eqnarray*}
After having introduced the mathematical framework for general conic linear optimisation problems, we are ready to use these results to get strong duality results for two particular cases of our general duality theory, which we discuss in the next two sections.

\section{Measures with $L^p$ Densities}\label{lpDensities}
Consider a special case of the primal problem $(P')$ from Section \ref{problemFormulation}, where the optimisation is over measures that have a density function that belongs to $L^p(\Phi)$.\\
Let $\Phi\subseteq\R^n$, $\Gamma\subseteq\R^m$ and $\Sigma\subseteq\R^k$.
\begin{align}
	(P^L) \sup_{f\in L^p(\Phi)} & \int_{\Phi} c(x)f(x)dx \nonumber\\ 
	                      & \int_{\Phi} A(y, x)f(x)dx \leq a(y), \text{ a.e. } y\in \Gamma \nonumber\\
	                      & \int_{\Phi} B(z, x)f(x)dx = b(z), \text{ a.e. } z\in \Sigma \nonumber\\
	                      & f \geq 0. \nonumber
\end{align}
Let $c\in L^q(\Phi)$ and $a\in L^p(\Psi)$ for some $1< p <\infty$ and $n, m, k$.\\
Assume $A:\Gamma\times\Phi\rightarrow \R$ and $B:\Sigma\times\Phi\rightarrow \R$ are such that 
\begin{eqnarray*}
A(y, \cdot)&\in L^q(\Phi)\\
A(\cdot, x)&\in L^p(\Gamma)\\
B(z, \cdot)&\in L^q(\Phi)\\
B(\cdot, x)&\in L^p(\Sigma)\\
\end{eqnarray*}
$\forall x\in\Phi, \forall y\in\Gamma$ and $\forall z\in\Sigma$.\\
Hence the functions $\tau_A(x) = \|A(\cdot, x)\|_p$, $\tau_B(x) = \|B(\cdot, x)\|_p$ and $\rho_A(y) = \|A(y, \cdot)\|_q$, $\rho_B(z) = \|B(z, \cdot)\|_q$ are well defined.\\
We also make the following assumption.
\begin{assumptions}\label{assumption:mappingToLp}
Functions $A$ and $B$ are such, that 
\begin{eqnarray*}
\int_{\Phi} A(y, x)f(x)dx\in L^p(\Gamma), \text{ a.e. }y\in\Gamma\\
\int_{\Phi} B(z, x)f(x)dx\in L^p(\Sigma), \text{ a.e. }z\in\Sigma
\end{eqnarray*}
\end{assumptions}
Later we will prove a lemma which, under some conditions, guarantees that the above assumptions are true, but for now we consider them as given.\\
Denote 
\begin{eqnarray*}
X &:=& L^p(\Phi),\\
C &:=& L^p_+(\Phi)\\
Y &:=& L^p(\Gamma)\times L^p(\Sigma),\\
K &:=& L^p_-(\Gamma)\times\left\{0\right\}.
\end{eqnarray*}
In this case, since $1< p < +\infty$ we have
\begin{eqnarray*}
X^* &:=& L^q(\Phi),\\
C^* &:=& L^q_+(\Phi)\\
Y^* &:=& L^q(\Gamma)\times L^q(\Sigma),\\
K^* &:=& L^q_-(\Gamma)\times L^q(\Sigma).
\end{eqnarray*}
and $X^{**} = X,\quad Y^{**} = Y,\quad  C^{**} = C,\quad  K^{**} = K$.\\
With these notations we can write the problem $(P^L)$ as
\begin{equation}
\min_{f\in C}\langle -c, f \rangle \quad \text{subject to} \quad \mathcal Af+h \in K,
\end{equation}
where the linear operator $\mathcal A:X\rightarrow Y$ is defined as
\begin{equation*}
\mathcal Af(y, z) = 
\left( \begin{array}{c}  \int_{\Phi} A(y, x)f(x)dx\\ 
\int_{\Phi} B(z, x)f(x)dx \end{array} \right)
\end{equation*}
and $h$ is defined as
\begin{equation*}
h(y, z) = 
\left( \begin{array}{c}  -a(y)\\ -b(z) \end{array} \right)
\end{equation*}
Construct the Lagrangian of the problem $(P^L)$
\begin{equation}
L(f, \lambda^*) = -\langle c, f\rangle + \langle \lambda^*, \mathcal Af+h\rangle,
\end{equation}
where $\lambda^*\in Y$, i.e. it has the following form
\begin{equation*}
\lambda^* = 
\left( \begin{array}{c}  g^*\\ s^* \end{array} \right), \quad g^*\in L^q(\Gamma), \quad s^*\in L^q(\Sigma)
\end{equation*}
The Lagrangian function can thus be written as
\begin{eqnarray}\label{lagrangianFunction}
L(f, g^*, s^*) &=& -\int_{\Phi}c(x)f(x)\diff x\nonumber\\
               &+& \int_{\Gamma}\left(\int_{\Phi}A(y, x)f(x)\diff x - a(y)\right)g^*(y)\diff y\\
               &+& \int_{\Sigma}\left(\int_{\Phi}B(z, x)f(x)\diff x - b(z)\right)s^*(z)\diff z\nonumber
\end{eqnarray}
The Lagrangian primal problem is
\begin{equation}
\min_{f\in C}\sup_{\lambda^*\in-K^*}L(f, \lambda^*)
\end{equation}
Interchanging the $\min$ and $\max$ operators, we obtain the dual Lagrangian, problem
\begin{equation}\label{lagrangianDualProblem}
\sup_{\lambda^*\in-K^*}\min_{f\in C}L(f, \lambda^*).
\end{equation}
In order to evaluate (\ref{lagrangianDualProblem}) we change the order of integration in (\ref{lagrangianFunction}). We have
\begin{equation}\label{lagrangianFunction2}
L(f, g^*, s^*) = -\langle a, g^*\rangle - \langle b, s^*\rangle + \langle \phi, f\rangle,
\end{equation}
where 
\begin{equation}
\phi(x) = \int_{\Gamma}A(y, x)g^*(y)\diff y + \int_{\Sigma}B(z, x)s^*(z)\diff z - c(x)
\end{equation}
Thus,
\begin{equation}
\min_{f\in C}L(f, g^*, s^*) =
\begin{cases}
-\langle a, g^*\rangle - \langle b, s^*\rangle, & \mbox{if } \phi(x) \geq 0, (\text{ a.e. } x\in\phi),\\ 
-\infty, & \mbox{otherwise}
\end{cases}
\end{equation}
This leads to the equivalence of the Lagrangian dual problem to
\begin{align}
	(D^L) \inf_{g\in L^q(\Gamma), s\in L^q(\Sigma)} & \int_{\Gamma} a(y)g(y)dy + \int_{\Sigma}b(z)s(z)\diff z\nonumber\\ 
	                      & \int_{\Gamma} A(y, x)g(y)dy + \int_{\Sigma} B(z, x)s(z)\diff z \geq c(x), \text{ a.e. } x\in \Phi \nonumber\\
	                      & g \geq 0. \nonumber
\end{align}
\begin{definition}
We say that Slater condition holds for the problem $(P^L)$ if $\exists\bar f\in L^p(\Phi)$ such that
\begin{eqnarray*}
\int_{\Phi}A(y,x)\bar f(x)dx < a(y), \quad a.e. \quad y \in \Gamma\\
\int_{\Phi}B(z,x)\bar f(x)dx = b(z), \quad a.e. \quad z \in \Sigma
\end{eqnarray*}
and the function $B$ is such that $\mathcal B(X) = L^p(\Sigma)$, where 
\begin{equation*}
\mathcal B(f) := \int_{\Phi}B(z,x)f(x)dx.
\end{equation*}
\end{definition}
Now we are ready to use the results from the previous section to prove strong duality in our case. Note that all the assumption in the Section \ref{sec:shapiro} are satisfied, Hence we get the following theorem.\\
\begin{theorem}[Strong Duality for $L^p$ Problems]
Let the linear operators $\mathcal A$ and $\langle c, \cdot \rangle$ be continuous. Suppose that Assumption (\ref{assumption:mappingToLp}) holds. Then, if the Slater condition holds for the problem $(P^L)$, we have $Val(P^L) = Val(D^L)$ and $Sol(D^L)$ is bounded.
\end{theorem}
Now we prove a lemma which ensures that Assumption \ref{assumption:mappingToLp} holds and the required continuity of the linear operators under some conditions:
\begin{lemma}\label{ShapiroProof:Lemma}
Suppose $A:\Gamma\times\Phi\rightarrow \R$ is such that $A(y, \cdot)\in L^q(\Phi)$ a.e. $y$ and $A(\cdot, x)\in L^p(\Gamma)$ a.e. $x$ and for some $1\leq p, q\leq\infty$ such, that $1/p + 1/q = 1$. Further suppose that the functions $\tau(x) := \|A(\cdot, x)\|_p$ and $\rho(y) := \|A(y, \cdot)\|_q$ are in $L^q(\Phi)$ and $L^p(\Gamma)$ respectively, and finally, that $\rho(y)$ is uniformly bounded, i.e. $\exists M$ such that $\rho(y) \leq M$ a.e. $y\in\Gamma$.\\
Then the linear operator $\mathcal A$ defined below is continuous, and its image is in $L^p(\Gamma)$:
\begin{equation}
\mathcal A(f) := \int_{\Phi}A(y,x)f(x)dx
\end{equation}
\end{lemma}
\begin{proof}
We have
\begin{eqnarray}
&&\int_{\Gamma}\left|\int_{\Phi} A(y, x)f(x)\diff x\right|^p\diff y\nonumber\\ \leq
&&\int_{\Gamma}\left(\int_{\Phi}\left| A(y, x)\right|\left|f(x)\right|\diff x\right)^p\diff y \label{mk_1}\\
\leq &&\left(\int_{\Phi}\left(\int_{\Gamma}\left|A(y, x)\right|^p\left|f(x)\right|^p\diff y\right)^{1/p}\diff x\right)^p \label{mk_2}\\
=&&\left(\int_{\Phi}\left|f(x)\right|\left(\int_{\Gamma}\left|A(y, x)\right|^p\diff y\right)^{1/p}\diff x\right)^p\nonumber\\ 
= && \left(\int_{\Phi}\left|f(x)\right|\tau(x)\diff x\right)^p\leq \infty\nonumber,
\end{eqnarray}
where \eqref{mk_2} follows from \eqref{mk_1} by Minkowski's inequality.\\
Now we prove continuity of $\mathcal A$.\\
For any $\epsilon\geq 0$ take $\delta = \epsilon / M$. Then for any two functions $f_1, f_2\in L^p(\Phi)$ such that $\|f_1 - f_2\|_p \leq \delta$ we have
\begin{eqnarray*}
&&\left\|\int_{\Phi} A(y, x)f_1(x)\diff x - \Irn A(y, x)f_2(x)\diff x\right\|_p =\\
&&\left\|\int_{\Phi} A(y, x)\left(f_1(x) - f_2(x)\right)\diff x\right\|_p \leq\\
&&\left\|f_1 - f_2\right\|_p\rho(y) \leq \epsilon.
\end{eqnarray*}
The last step follows from H\"{o}lder's inequality.
\end{proof}

\newpage
\section{Semi-Infinite Problems with Constraints via Bounds on Integrals of Piece-wise Continuous Functions}\label{semiInf}
In this section we discuss strong duality of semi-infinite programming problems with specific structure.\\
Consider the following optimisation problem over spaces of measures:
\begin{align*}
\text{(P)}\quad\sup_{\calF\in\cM_{\fF}}\,&\int_{\Phi}h(x)\diff \calF(x)\\
\text{s.t. }&\int_{\Phi}\phi_{s}(x)\diff\calF(x)
\leq a_{s},\quad(s=1,\dots,M),\\
\text{s.t. }&\int_{\Phi}\psi_{t}(x)\diff\calF(x)
= b_{t},\quad(t=1,\dots,N),\\
&\calF\geq 0,
\end{align*}
and it's dual problem
\begin{align}
\text{(D)}\quad\inf_{(y,z)\in\R^{M+N+1}}
\,&\sum_{s=1}^{M}a_s y_s+\sum_{t=1}^{N}b_t z_t,\nonumber\\
\text{s.t. }&\sum_{s=1}^{M}y_s\phi_s(x)+\sum_{t=1}^{N}z_t\psi_t(x)-h(x)\geq 0,\quad(x\in\Phi),\label{pos10}\\
&y\geq 0,\nonumber
\end{align}
where $(\Phi,\fF)$ is a complete measure space. Let $\cM_{\fF}$ be the set of signed measures with finite variation on $(\Phi,\fF)$.\\
Shapiro, \cite{shap}, proves that strong duality holds (i.e., $Val(P) = Val(D)$) when $\Phi$ is compact and the functions $h(x)$, $\phi_s(x)$ and $\psi_t(x)$ are continuous. In this section we extend this result to the case where these functions are piecewise continuous on the partitioning of $\Phi$ into boxes $\Phi=\bigcup_{i=1}^k\mathbb B_i$, $\cap_{i=1}^k\mathbb B_i = \emptyset$. Each box $\mathbb B_i\subset \R^n$ has the following form
\begin{equation*}
\mathbb B_i = \left\{x\in\R^n : l_j^i \leq x_j < u_j^i \right\}
\end{equation*}
Suppose that each of the functions $h(x)$, $\phi_s(x)$ and $\psi_t(x)$ is continuous on $\mathbb B_k$, $\forall k\leq K$.\\
We take a similar approach as in \cite{shap}.\\
Note that each box in $\R^n$ can be linearly transformed into a unit box in $\R^n$, so that with the new transformed variables the optimisation problem becomes
\begin{align*}
\text{(P')}\quad\sup_{\calF\in\cM_{\fF}}\,&\sum_{i=1}^K\int_{\mathbb B}h^i(x)\diff \calF^i(x)\\
\text{s.t. }&\sum_{i=1}^K\int_{\mathbb B}\phi_{s}^i(x)\diff\calF_i(x)
\leq a_{s},\quad(s=1,\dots,M),\\
\text{s.t. }&\sum_{i=1}^K\int_{\mathbb B}\psi_{t}^i(x)\diff\calF_i(x)
= b_{t},\quad(t=1,\dots,N),\\
&\calF_i\geq 0, \quad i = 1, \dots, K.
\end{align*}
with the new dual (which, obviously, is equivalent to the original dual problem)
\begin{align}
\text{(D')}\quad\inf_{(y,z)\in\R^{M+N}}
\,&\sum_{s=1}^{M}a_s y_s+\sum_{t=1}^{N}b_t z_t\nonumber\\
\text{s.t. }&\sum_{s=1}^{M}y_s\phi_s^i(x)+\sum_{t=1}^{N}z_t\psi_t^i(x)-h^i(x)\geq 0,\quad(x\in\mathbb B, \quad i = 1, \dots, K),\label{pos10}\\
&y\geq 0,\nonumber
\end{align}
where $\phi_s^i(x) = \phi_s|_{\mathbb B_i}(T_i^-(x))$ and $T_i:\mathbb B_i\rightarrow\mathbb B$ is the linear map of the transformation ($T_i^-$ being it's inverse map).\\
Let $X:= \R^{M+N}$ and $Y:=\times_{i=1}^K \mathcal C(\mathbb B)$, where $\mathcal C(\mathbb B)$ is the set of continuous functions on $\mathbb B$. Then the problem $(D')$ can be written as
\begin{align}\label{semi-inf_gen}
\inf_{(y, z)\in\R^M_+\times\R^N}&\sum_{s=1}^{M}a_s y_s+\sum_{t=1}^{N}b_t z_t\nonumber\\
\text{s.t. }& \mathcal A\cdot(y, z) + b\in \mathcal K,
\end{align}
where $\mathcal A:X\rightarrow Y$ is defined as
\begin{equation}
\mathcal A\cdot(y, z) = \left(\tau_1(y, z, x), \dots, \tau_K(y, z, x)\right),
\end{equation}
$\tau_i(y, z, x) = \sum_{s=1}^{M}y_s\phi_s^i(x)+\sum_{t=1}^{N}z_t\psi_t^i(x)$ and $b$ is defined as vector
\begin{equation}
b = \left(h^i(x), \dots, h^K(x)\right).
\end{equation}
For cone $\mathcal K$ we have $\mathcal K = \times_{i=1}^K \mathcal C_+(\mathbb B)$, where $\mathcal C_+(\mathbb B)$ is the set of non-negative continuous functions on $\mathbb B$.\\
The dual space $Y^*$ of $Y$ is $\times_{i=1}^K\mathcal M$, where $\mathcal M$ is the set of finite signed Borel measures on $\mathbb B$. By equipping $Y$ and $Y^*$ with the strong and weak topologies respectively, and by defining the scalar product between $Y$ and $Y^*$ as
\begin{equation}
\langle \phi, \mu\rangle := \sum_{i=1}^K\int_{\mathbb B}\phi_i(x)\diff \mu_i(x),
\end{equation}
we obtain a pair of locally convex topological vector spaces.\\
It is easy to see that the Lagrangian dual of the Problem 
(\ref{semi-inf_gen}) coincides with the problem $(P')$. The following proposition is a direct result of Proposition \ref{mainProp}.
\begin{proposition}
Suppose that the optimal value of the problem $(D)$ is finite. If there exists $(\bar y, \bar z)\in\R^M\times\R^N$ such that
\begin{equation}
\sum_{s=1}^{M}\bar y_s\phi_s(x)+\sum_{t=1}^{N}\bar z_t\psi_t(x)-h(x)> 0, \quad \forall x\in \Phi,
\end{equation}
then $Val(P) = Val(D)$ and the set of optimal solution of the problem $(P)$ is bounded.
\end{proposition}


\newpage

\end{document}